\documentclass[12pt,draft]{scrartcl}

\usepackage[]{amsmath, amssymb,amsfonts,amsthm,mathtools,braket,color,enumerate}

\usepackage{here}
\usepackage{tikz-cd}
\usepackage{tikz}
\usetikzlibrary{decorations}
\usetikzlibrary{arrows}
\tikzstyle{v} = [circle, draw, inner sep=2pt, minimum size=3pt, fill=black]
\tikzstyle{e} = [fill, opacity=.2, 
fill opacity=.2, 
line join=round, 
line width=13pt]
\pgfdeclarelayer{background}
\pgfsetlayers{background,main}

\pgfdeclaredecoration{dashsoliddouble}{initial}{
  \state{initial}[width=\pgfdecoratedinputsegmentlength]{
    \pgfmathsetlengthmacro\lw{.3pt+.5\pgflinewidth}
    \begin{pgfscope}
      \pgfpathmoveto{\pgfpoint{0pt}{\lw}}%
      \pgfpathlineto{\pgfpoint{\pgfdecoratedinputsegmentlength}{\lw}}%
      \pgfmathtruncatemacro\dashnum{%
        round((\pgfdecoratedinputsegmentlength-3pt)/6pt)
      }
      \pgfmathsetmacro\dashscale{%
        \pgfdecoratedinputsegmentlength/(\dashnum*6pt + 3pt)
      }
      \pgfmathsetlengthmacro\dashunit{3pt*\dashscale}
      \pgfsetdash{{\dashunit}{\dashunit}}{0pt}
      \pgfusepath{stroke}
      \pgfsetdash{}{0pt}
      \pgfpathmoveto{\pgfpoint{0pt}{-\lw}}%
      \pgfpathlineto{\pgfpoint{\pgfdecoratedinputsegmentlength}{-\lw}}%
      \pgfusepath{stroke}
    \end{pgfscope}
  }
}

\theoremstyle{plain}
\newtheorem{theorem}{Theorem}[section]
\newtheorem{lemma}[theorem]{Lemma}
\newtheorem{proposition}[theorem]{Proposition}

\theoremstyle{definition}
\newtheorem{definition}[theorem]{Definition}

\newtheorem{remark}[theorem]{Remark}
\newtheorem{case}{Case}

\DeclareMathOperator{\Der}{Der}
\DeclareMathOperator{\rank}{rank}
\DeclareMathOperator{\Ker}{Ker}

\title {Freeness of Hyperplane Arrangements between Boolean Arrangements and Weyl Arrangements of Type $ B_{\ell} $
}
\author{
Michele Torielli\thanks{Department of Mathematics, GI-CoRE GSB, Hokkaido University, Sapporo, Hokkaido 060-0810, Japan.
email:torielli@math.sci.hokudai.ac.jp},
Shuhei Tsujie\thanks{Department of Information Design, Hiroshima Kokusai Gakuin University, Hiroshima 739-0321, Japan.
email:tsujie@edu.hkg.ac.jp}
}
\date{}

\begin{document}
\maketitle
	
\begin{abstract}
Every subarrangement of Weyl arrangements of type $ B_{\ell} $ is represented by a signed graph. 
Edelman and Reiner characterized freeness of subarrangements between type $ A_{\ell-1} $ and type $ B_{\ell} $ in terms of graphs. 
Recently, Suyama and the authors characterized freeness for subarrangements containing Boolean arrangements satisfying a certain condition. 
This article is a sequel to the previous work. 
Namely, we give a complete characterization for freeness of arrangements between Boolean arrangements and Weyl arrangements of type $ B_{\ell} $ in terms of graphs. 
\end{abstract}



\section{Introduction}\label{sec:introduction}
A (central) hyperplane arrangement in a vector space is a finite collection of vector subspaces of codimension one. 
In this article, we are mainly interested in the study of \textbf{freeness} of arrangements (see Section \ref{sec:preliminaries} for definitions and basic properties). 
An arrangement consisting of the reflecting hyperplanes of a Weyl group is called a \textbf{Weyl arrangement}. 
Saito \cite{saito1977uniformization-rk,saito1980theory-jotfostuotsam} proved that every Weyl arrangement is free. 
However, no complete characterizations of freeness for Weyl subarrangements are known except for type $ A $. 
Weyl subarrangements of type $ A $ are represented by simple graphs and their freeness is characterized in terms of simple graphs (see Section \ref{sec:preliminaries}). 
In a similar way, Weyl subarrangements of type $ B $ are represented by signed graphs. 

In this article, a \textbf{signed graph} is a pair $ G=(G^{+},G^{-}) $ in which $ G^{+} = (V_{G}, E_{G}^{+}) $ and $ G^{-}=(V_{G},E_{G}^{-}) $ are simple graphs on a common set of vertices $ V_{G} $. 
Notice that we do not consider graphs with loops, half edges, and so on. 
See \cite{zaslavsky1982signed-dam} for a general treatment of signed graphs. 
Elements in the set $ E_{G}^{+} $ (respectively $ E_{G}^{-} $) are called \textbf{positive} (respectively \textbf{negative}) \textbf{edges}. Edges are sometimes called \textbf{links}. 

Let $ \mathbb{K} $ be a field of characteristic zero. 
For each signed graph $ G $ on the vertex set $ \{1, \dots, \ell\} $, we define the \textbf{signed-graphic arrangement} $ \mathcal{A}(G) $ in the $ \ell $-dimensional vector space $ \mathbb{K}^{\ell} $ by 
\begin{multline*}
\mathcal{A}(G) 
\coloneqq  \Set{\{x_{i}=0\} | 1 \leq i \leq \ell} \\ \cup \Set{\{x_{i}-x_{j}=0\} | \{i,j\} \in E_{G}^{+}} 
\cup \Set{\{x_{i}+x_{j}=0\} | \{i,j\} \in E_{G}^{-}},
\end{multline*}
where $ (x_{1}, \dots, x_{\ell}) $ denotes a basis of the dual space $ (\mathbb{K}^{\ell})^{\ast} $ and, for each linear form $ \alpha \in (\mathbb{K}^{\ell})^{\ast} $, $ \{\alpha = 0 \} $ denotes the hyperplane $ \Ker (\alpha) $ in $ \mathbb{K}^{\ell} $. 
Note that, in \cite{bailey????inductively-p, edelman1994free-mz,  suyama2019signed-dm}, the authors considered signed graphs with loops and associated to each loop the corresponding coordinate hyperplane. 
In this article, however, we will always assume that every signed-graphic arrangements contain all the coordinate hyperplanes. 

Edelman and Reiner \cite[Theorem 4.6]{edelman1994free-mz} characterized the freeness of Weyl subarrangements between type $ A_{\ell-1} $ and $ B_{\ell} $ in terms of graphs. 
Bailey \cite{bailey????inductively-p} characterized the freeness of signed-graphic arrangements for some cases. 
Suyama and the authors characterized the freeness of signed-graphic arrangements corresponding to graphs in the case $ G^{+} \supseteq G^{-} $ as follows. 
\begin{theorem}[{\cite[Theorem 1.4]{suyama2019signed-dm}}]\label{previous theorem}
Let $ G $ be a signed graph with $ G^{+} \supseteq G^{-} $. 
Then the following conditions are equivalent: 
\begin{enumerate}[(1)]
\item $ G $ is balanced chordal. 
\item\label{previous theorem 2} $ \mathcal{A}(G) $ is divisionally free. 
\item $ \mathcal{A}(G) $ is free. 
\end{enumerate}
\end{theorem}
Note that condition (\ref{previous theorem 2}) in Theorem \ref{previous theorem} is omitted in \cite[Theorem 1.4]{suyama2019signed-dm}. 
However, the verification of freeness was obtained via Abe's division theorem, and therefore condition (\ref{previous theorem 2}) is also equivalent to the other two conditions. 

The main result of this article is a generalization of Theorem \ref{previous theorem} as follows. 
Note that a signed-graphic arrangement $ \mathcal{A}(G) $ is a Weyl subarrangement of type $ B $ containing a Boolean arrangement, and vice versa, 
where a Boolean arrangement is an arrangement consisting of the coordinate hyperplanes $ \{x_{i}=0\} $. 

\makeatletter
\renewcommand{\p@enumii}{}
\makeatother
\begin{theorem}\label{main theorem}
Let $ G $ be a signed graph. 
Then the following conditions are equivalent: 
\begin{enumerate}[(1)]
\item\label{main theorem 1} $ G $ satisfies the following three conditions: 
\begin{enumerate}[(I)]
\item \label{I} $ G $ is balanced chordal. 
\item \label{II} $ G $ has no induced subgraphs isomorphic to unbalanced cycles of length three or more. 
\item \label{III} $ G $ has no induced subgraphs which are switching equivalent to the graph in Figure \ref{Fig:obstruction}. 
\begin{figure}[t]
\centering
\begin{tikzpicture}
\draw (0,1) node[v](1){}; 
\draw (0,0) node[v](2){}; 
\draw (1,0) node[v](3){};
\draw (1,1) node[v](4){};  
\draw[decoration={dashsoliddouble}, decorate] (2)--(1);
\draw[decoration={dashsoliddouble}, decorate] (4)--(3);
\draw (1)--(4)--(2)--(3);
\draw[dashed] (1)--(3);
\end{tikzpicture}
\caption{An obstruction to freeness (Dashed line segments denote negative edges)}\label{Fig:obstruction}
\end{figure}
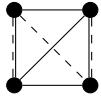
\end{enumerate}
\item\label{main theorem 2} $ \mathcal{A}(G) $ is divisionally free. 
\item\label{main theorem 3} $ \mathcal{A}(G) $ is free. 
\end{enumerate}
\end{theorem}

The organization of this article is as follows. 
In Section \ref{sec:preliminaries}, we review basic properties of freeness of hyperplane arrangements and we recall several results on graphic and signed-graphic arrangements. 
Moreover, we present necessary conditions for the freeness of a signed-graphic arrangement and we prove the implication $ (\ref{main theorem 3}) \Rightarrow (\ref{main theorem 1}) $ of Theorem \ref{main theorem}. 
In Section \ref{sec:tildeG}, we introduce a way to construct another signed graph $ \tilde{G} $ from a signed graph $ G $. 
The signed graph $ \tilde{G} $ satisfies the property $ \tilde{G}^{+} \supseteq \tilde{G}^{-} $ and hence we may apply Theorem \ref{previous theorem} to $ \tilde{G} $. 
In Section \ref{sec:tildeG is complete}, we prove Theorem \ref{main theorem} in the case $ \tilde{G}^{+} $ is complete. 
In Section \ref{sec:BC of tildeG}, we prove a key lemma which states that if $ G $ satisfies conditions (\ref{I})(\ref{II})(\ref{III}), then $ \tilde{G} $ is balanced chordal. This section is an essential part of this article. 
In Section \ref{sec:PROOF}, we prove the implication $ (\ref{main theorem 1}) \Rightarrow (\ref{main theorem 2}) $ of Theorem \ref{main theorem} and complete the proof of our main result.

\section{Preliminaries}\label{sec:preliminaries}
\subsection{Hyperplane arrangements}

Let $ \mathbb{K} $ be an arbitrary field, $ \mathcal{A} $ a central arrangement in the $ \ell $-dimensional vector space $ \mathbb{K}^{\ell} $. 
Our main reference on the theory of hyperplane arrangement is \cite{orlik1992arrangements}. 
Define the set $ L(\mathcal{A}) $ by 
\begin{align*}
L(\mathcal{A}) \coloneqq \Set{\bigcap_{H\in \mathcal{B}}H | \mathcal{B} \subseteq \mathcal{A}}. 
\end{align*}
Note that when $ \mathcal{B} $ is empty, the intersection over $ \mathcal{B} $ is the ambient vector space $ \mathbb{K}^{\ell} $. 
For subspaces $ X, Y \in L(\mathcal{A}) $, we define a partial order $ \leq $ by the reverse inclusion, that is, $ X \leq Y \stackrel{\text{def}}{\Leftrightarrow} X \supseteq Y $. 
The set $ L(\mathcal{A}) $ equipped with the partial order $ \leq $ forms a geometric lattice. 
We call $ L(\mathcal{A}) $ the \textbf{intersection lattice} of $ \mathcal{A} $. 
The \textbf{rank} of $ \mathcal{A} $, denoted by $ \rank(\mathcal{A}) $, is the codimension of the center $ \bigcap_{H \in \mathcal{A}}H $. 
The one-variable M\"{o}ebius function $ \mu \colon L(\mathcal{A}) \to \mathbb{Z} $ is defined recursively by 
\begin{align*}
\mu(\mathbb{K}^{\ell}) \coloneqq 1 
\quad \text{ and } \quad 
\mu(X) \coloneqq -\sum_{Y < X}\mu(Y). 
\end{align*}
The \textbf{characteristic polynomial } $\chi(\mathcal{A},t) $  of $\mathcal{A}$ is defined by 
\begin{align*}
\chi(\mathcal{A},t) \coloneqq \sum_{X \in L(\mathcal{A})}\mu(X)t^{\dim X}. 
\end{align*}

Let $ S $ be the ring of polynomial functions on $ \mathbb{K}^{\ell} $ and $ \Der(S) $ the module of derivations of $ S $. 
Namely, 
\begin{align*}
\Der(S) \coloneqq \Set{\theta \colon S \to S | \text{ $ \theta $ is $ \mathbb{K} $-linear and } \theta(fg) = \theta(f)g + f\theta(g) \text{ for any } f,g \in S}. 
\end{align*}
Define the \textbf{module of logarithmic derivations} by 
\begin{align*}
D(\mathcal{A}) \coloneqq \Set{\theta \in \Der(S) | \theta(\alpha_{H}) \in \alpha_{H}S \text{ for any } H \in \mathcal{A}}, 
\end{align*}
where $ \alpha_{H} \in (\mathbb{K}^{\ell})^{\ast} $ denotes a defining linear form of a hyperplane $ H \in \mathcal{A} $. 
Note that $ D(\mathcal{A}) $ is a graded $ S $-module. 
\begin{definition}
An arrangement $ \mathcal{A} $ is said to be \textbf{free} if $ D(\mathcal{A}) $ is a free $ S $-module. 
\end{definition}

For each hyperplane $ H \in \mathcal{A} $, we define the \textbf{restriction} $ \mathcal{A}^{H} $ by 
\begin{align*}
\mathcal{A}^{H} &\coloneqq \Set{H \cap K | K \in \mathcal{A}\setminus\{H\}}. 
\end{align*}
Note that $ \mathcal{A}^{H} $ is an arrangement in $ H $. 

There are several beautiful results on free arrangements. The following results played an important role in \cite{suyama2019signed-dm}.

\begin{theorem}[Abe {\cite[Theorem 1.1]{abe2016divisionally-im}} (Division Theorem)]\label{Abe division}
Let $ H \in \mathcal{A} $. 
Suppose that $ \mathcal{A}^{H} $ is free and $ \chi(\mathcal{A}^{H},t) $ divides $ \chi(\mathcal{A},t) $. 
Then $ \mathcal{A} $ is free. 
\end{theorem}
Abe's division theorem leads to the notion of divisional freeness. 
\begin{definition}
\textbf{Divisional freeness} is defined recursively by the following rules. 
\begin{enumerate}[(1)]
\item The empty arrangements are divisionally free. 
\item If there exists $ H \in \mathcal{A} $ such that $ \mathcal{A}^{H} $ is divisionally free and $ \chi(\mathcal{A}^{H},t) $ divides $ \chi(\mathcal{A},t) $, then $ \mathcal{A} $ is divisionally free. 
\end{enumerate}
\end{definition}

Another important notion in the theory of arrangements is supersolvability. We say that $ \mathcal{A} $ is \textbf{supersolvable} if the intersection lattice $ L(\mathcal{A}) $ is supersolvable, that is, there exists a maximal chain of $ L(\mathcal{A}) $ consisting of modular elements (see \cite{stanley1972supersolvable-au} for more details). 
It is well known that every supersolvable arrangement is (divisionally) free (see \cite[Theorem 4.2]{jambu1984free-aim} and \cite[Theorem 4.4(2)]{abe2016divisionally-im} ). 

\subsection{Simple graphs and graphic arrangements}
Let $ G=(V_{G},E_{G}) $ be a simple graph. 
Given a subset $ W \subseteq V_{G} $, let $ G[W] $ denote the induced subgraph on $ W $. 
A \textbf{chord} of a cycle in $ G $ is an edge connecting two non-consecutive vertices of the cycle. 
We say that $ G $ is \textbf{chordal} if every cycle of length four or more has a chord. 
A vertex of $ G $ is called \textbf{simplicial} if its neighborhood is a clique, that is, a complete subgraph. 
An ordering $ (v_{1}, \dots, v_{\ell}) $ of the vertices of $ G $ is called a \textbf{perfect elimination ordering} if $ v_{i} $ is simplicial in the induced subgraph $ G[\{v_{1}, \dots, v_{i}\}] $ for each $ i \in \{1, \dots, \ell\} $. 

For a simple graph $ G $ on the vertex set $ \{1,\dots, \ell\} $, we define the \textbf{graphic arrangement} $ \mathcal{A}(G) $ in $ \mathbb{K}^{\ell} $ by 
\begin{align*}
\mathcal{A}(G) \coloneqq \Set{\{x_{i}-x_{j}=0 \} | \{i,j\} \in E_{G}}. 
\end{align*}
Note that if we identify a simple graph $ G $ on $ \ell $ vertices with a signed graph $ (G, \overline{K}_{\ell}) $, where $ \overline{K}_{\ell} $ denotes the edgeless graph on $ \ell $ vertices, then the graphic arrangement $ \mathcal{A}(G) $ coincides with the signed-graphic arrangement $ \mathcal{A}(G,\overline{K}_{\ell}) $. 

Characterizations of freeness of graphic arrangements are well known as follows. 
\begin{theorem}[Edelman-Reiner {\cite[Theorem 3.3]{edelman1994free-mz}}, Fulkerson-Gross {\cite[Section 7]{fulkerson1965incidence-pjom}}, Stanley {\cite[Corollary 4.10]{stanley2004introduction-lnicmi}}]
Let $ G $ be a simple graph. 
Then the following conditions are equivalent: 
\begin{enumerate}[(1)]
\item $ G $ is chordal. 
\item $ G $ has a perfect elimination ordering. 
\item $ \mathcal{A}(G) $ is supersolvable. 
\item $ \mathcal{A}(G) $ is free. 
\end{enumerate}
\end{theorem}

Together with chordal graphs, there is another important family of graphs.
\begin{definition}
\textbf{Threshold graphs} are defined recursively by the following conditions: 
\begin{enumerate}[(1)]
\item The single-vertex graph $ K_{1} $ is threshold. 
\item If $ G $ is threshold, then the graph obtained by adding an isolated vertex to $ G $ is threshold. 
\item If $ G $ is threshold, then the graph obtained by adding a dominating vertex to $ G $ is threshold, where a vertex is said to be \textbf{dominating} if it is adjacent to the other vertices. 
\end{enumerate}
\end{definition}

The \textbf{degree} of a vertex $ v $ of $ G $, denoted $ \deg_{G}(v) $, is the number of vertices adjacent to $ v $. 
An ordering $ (v_{1}, \dots, v_{\ell}) $ of the vertices of $ G $ is called a \textbf{degree ordering} if $ \deg_{G}(v_{1}) \geq \dots \geq \deg_{G}(v_{\ell}) $. 
Note that if $ G $ is threshold, then we have either $ v_{1} $ is dominating or $ v_{\ell} $ is isolated in every degree ordering. 

Threshold graphs can also be characterized by forbidden induced subgraphs. 
\begin{theorem}[Golumbic {\cite[Corollary 5]{golumbic1978trivially-dm}}]\label{forbidden induced subgraphs threshold}
A simple graph $ G $ is threshold if and only if $ G $ is $ (2K_{2}, C_{4}, P_{4}) $-free, that is, $ G $ has no induced subgraphs isomorphic to these three graphs (see Figure \ref{Fig:threshold}). 
\end{theorem}
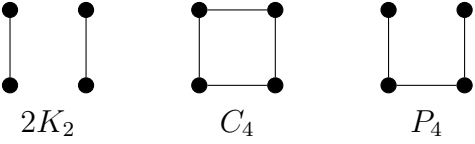
\begin{figure}[t]
\centering
\begin{tikzpicture}
\draw (0,1) node[v](1){};
\draw (0,0) node[v](2){};
\draw (1,0) node[v](3){};
\draw (1,1) node[v](4){};
\draw (0.5,-0.5) node[]{$ 2K_{2} $};
\draw (1)--(2);
\draw (3)--(4);
\end{tikzpicture}
\hspace{10mm}
\begin{tikzpicture}
\draw (0,1) node[v](1){};
\draw (0,0) node[v](2){};
\draw (1,0) node[v](3){};
\draw (1,1) node[v](4){};
\draw (0.5,-0.5) node[]{$ C_{4} $};
\draw (1)--(2)--(3)--(4)--(1);
\end{tikzpicture}
\hspace{10mm}
\begin{tikzpicture}
\draw (0,1) node[v](1){};
\draw (0,0) node[v](2){};
\draw (1,0) node[v](3){};
\draw (1,1) node[v](4){};
\draw (0.5,-0.5) node[]{$ P_{4} $};
\draw (1)--(2)--(3)--(4);
\end{tikzpicture}
\caption{Forbidden induced subgraphs for threshold graphs}\label{Fig:threshold}
\end{figure} 

\subsection{Signed graphs and signed-graphic arrangements}
Given a signed graph $ G $ and a subset $ W \subseteq V_{G} $, we call $ G[W] \coloneqq (G^{+}[W],G^{-}[W]) $ the \textbf{induced subgraph} on $ W $. 
For a vertex $ v \in V_{G} $, let $ G\setminus v $ denote the induced subgraph on $ V_{G}\setminus \{v\} $. 
Call an edge $ e \in E_{G}^{+} \cup E_{G}^{-} $ \textbf{double} if $ e \in E_{G}^{+} \cap E_{G}^{-} $, otherwise call $ e $ \textbf{single}. 
A \textbf{path} of $ G $ is a sequence of distinct vertices $ v_{1}, \dots, v_{k} $ with positive or negative edges $ \{v_{i}, v_{i+1}\}  \quad (1 \leq i \leq k-1) $. 
For $ k \geq 3 $, a \textbf{cycle} of length $ k $ of $ G $ is a path $ v_{1}, \dots, v_{k} $ with an additional edge $ \{v_{1}, v_{k} \} $. 
We use brackets for describing the negative parts of paths and cycles (see Figure \ref{Fig:example of notation of paths} for examples). 
\begin{figure}[t]
\centering
\begin{tikzpicture}
\draw (0,0) node[v,label=above:{$ a $}](1){};
\draw (1,0) node[v,label=above:{$ b $}](2){};
\draw (2,0) node[v,label=above:{$ c $}](3){};
\draw (3,0) node[v,label=above:{$ d $}](4){};
\draw (1)--(2)--(3)--(4);
\draw (1.5,-.5) node(5){$ abcd $};
\end{tikzpicture}
\qquad
\begin{tikzpicture}
\draw (0,0) node[v,label=above:{$ a $}](1){};
\draw (1,0) node[v,label=above:{$ b $}](2){};
\draw (2,0) node[v,label=above:{$ c $}](3){};
\draw (3,0) node[v,label=above:{$ d $}](4){};
\draw[dashed] (1)--(2)--(3);
\draw (3)--(4);
\draw (1.5,-.5) node(5){$ [abc]d $};
\end{tikzpicture}
\qquad
\begin{tikzpicture}
\draw (0,0) node[v,label=above:{$ a $}](1){};
\draw (1,0) node[v,label=above:{$ b $}](2){};
\draw (2,0) node[v,label=above:{$ c $}](3){};
\draw (3,0) node[v,label=above:{$ d $}](4){};
\draw[dashed] (1)--(2);
\draw (2)--(3);
\draw[dashed] (3)--(4);
\draw (1.5,-.5) node(5){$ [ab][cd] $};
\end{tikzpicture}
\caption{Examples of notation of paths}\label{Fig:example of notation of paths}
\end{figure}
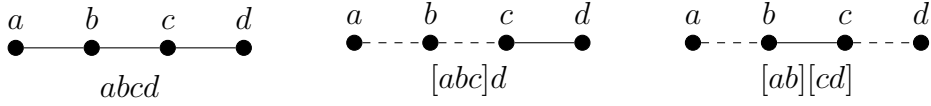

A \textbf{chord} of a cycle is an edge between two non-consecutive vertices in the cycle. 
A cycle $ C $ is called \textbf{balanced} if $ C $ has an even number of negative edges, otherwise $ C $ is called \textbf{unbalanced}. 
A \textbf{balanced chord} of a balanced cycle $ C $ is a chord which separates $ C $ into two balanced cycles. 
\begin{definition}
A signed graph $ G $ is said to be \textbf{balanced chordal} if every balanced cycle of length at least four has a balanced chord. 
\end{definition}
Balanced chordality is a generalization of chordality of simple graphs. 
Namely, if $ G $ is a chordal simple graph on $ \ell $ vertices, then the signed graph $ (G, \overline{K}_{\ell}) $ is balanced chordal. 
We have the following proposition describing the relation of such graphs with freeness. 
\begin{proposition}[{\cite[Lemma 4.7]{suyama2019signed-dm}}]\label{free => I}
If a signed-graphic arrangement $ \mathcal{A}(G) $ is free, then $ G $ is balanced chordal. 
\end{proposition}
\begin{proposition}[{\cite[Proposition 4.6]{suyama2019signed-dm}}]\label{free => locally free}
Let $ F $ be an induced subgraph of a signed graph $ G $. 
If  $ \mathcal{A}(G) $ is free, then $ \mathcal{A}(F) $ is free. 
\end{proposition}

As for the concept of balanced chordal signed graphs, also the notion of simplicial vertices can be generalized to signed graphs.
\begin{definition}
A vertex $ v $ of $ G $ is said to be \textbf{link simplicial} if, for any two distinct edges incident to $ v $, there exists an edge such that they form a balanced triangle. 
An ordering $ (v_{1} \dots, v_{\ell}) $ of the vertices of $ G $ is called a \textbf{link elimination ordering} if $ v_{i} $ is link simplicial in the induced subgraph $ G[\{v_{1}, \dots, v_{i}\}] $ for each $ i \in \{1, \dots, \ell\} $. 
\end{definition}
\begin{remark}
Note that a vertex is link simplicial if and only if it is bias simplicial (or signed simplicial) in the signed graph with a loop at every vertex (see \cite{zaslavsky2001supersolvable-ejoc, suyama2019signed-dm}). 
\end{remark}

Zaslavsky completely characterized supersolvability of signed-graphic arrangements. 
However, we need only the following result in this article. 
\begin{theorem}[Zaslavsky {\cite[Theorem 2.2]{zaslavsky2001supersolvable-ejoc}}]\label{zaslavsky ss}
Suppose that a signed graph $ G $ has a link elimination ordering. 
Then $ \mathcal{A}(G) $ is supersolvable. 
\end{theorem}

A \textbf{switching function} $ \nu $ of a signed graph $ G $ is a function $ \nu \colon V_{G} \to \{\pm 1 \} $. 
The signed graph switched by $ \nu $ is a signed graph $ G^{\nu} $ consisting of the following data: 
\begin{enumerate}[(1)]
\item $ V_{G^{\nu}} \coloneqq V_{G} $. 
\item $ E_{G^{\nu}}^{+} \coloneqq \Set{\{u,v\} \in E_{G}^{+} | \nu(u) = \nu(v)} \cup \Set{\{u,v\} \in E_{G}^{-} | \nu(u) \neq \nu(v)} $. 
\item $ E_{G^{\nu}}^{-} \coloneqq \Set{\{u,v\} \in E_{G}^{+} | \nu(u) \neq \nu(v)} \cup \Set{\{u,v\} \in E_{G}^{-} | \nu(u) = \nu(v)} $. 
\end{enumerate}
We say that $ G $ and $ G^{\nu} $ are \textbf{switching equivalent}. 
\begin{remark}
Note that freeness and the intersection lattice of signed-graphic arrangements is stable under switching since switching acts on arrangements as a coordinate transformation. 
Moreover, conditions (\ref{I}) (\ref{II}) (\ref{III}) from Theorem~\ref{main theorem} are also stable under switching. 
\end{remark}

Let $ e $ be an edge of a signed graph $ G $ and $ H = \{x_{i}\pm x_{j}=0\} $ the corresponding hyperplane (the sign depends on the sign of $ e $). 
Deleting $ x_{i} $ in the defining equations of hyperplanes in $ \mathcal{A}(G) $ with the relation $ x_{i}\pm x_{j}=0 $ we obtain the restriction $ \mathcal{A}(G)^{H} $. 
This is again a signed-graphic arrangement, that is, there exists a signed graph $ G_{i} $ on $ V_{G}\setminus \{i\} $ such that $ \mathcal{A}(G_{i}) = \mathcal{A}(G)^{H} $. 
If we delete $ x_{j} $, then we have another representation of $ \mathcal{A}(G)^{H} $ and there exists a signed graph $ G_{j} $ on $ V_{G}\setminus \{j\} $ such that $ \mathcal{A}(G_{j}) = \mathcal{A}(G)^{H} $. 
The graphs $ G_{i} $ and $ G_{j} $ are not isomorphic in general but one can show that they are switching equivalent. 
Let $ G/e $ denote $ G_{i} $ or $ G_{j} $ and call it the \textbf{contraction}. 

\begin{proposition}\label{free => II}
Let $ G $ be an unbalanced cycle of length three or more. 
Then $ \mathcal{A}(G) $ is non-free. 
\end{proposition}
\begin{proof}
We proceed by induction on the length $ \ell $ of the cycle $ G $. 
First assume that $ \ell = 3 $. 
Then the characteristic polynomial is 
\begin{align*}
\chi(\mathcal{A}(G),t) = t^{3} - 6 t^{2} + 12 t - 7 = (t - 1) (t^{2} - 5 t + 7). 
\end{align*}
By Terao's factorization theorem \cite{terao1981generalized-im}, we have $ \mathcal{A}(G) $ is non-free. 

Now suppose that $ \ell \geq 4 $. 
Assume that $ \mathcal{A}(G) $ is free.
Fix an edge $ e $ and let $ H $ be the corresponding hyperplane. 
Since $ G \setminus \{e\} $ has a link elimination ordering, the deletion $ \mathcal{A}(G) \setminus \{H\} $ is supersolvable by Theorem \ref{zaslavsky ss}. 
Using the restriction theorem \cite[Corollary 4.47]{orlik1992arrangements}, we have that $ \mathcal{A}(G)^{H} $ is free. 
However, since $ G/e $ is an unbalanced cycle of length $ \ell -1 $, by the induction hypothesis, $ \mathcal{A}(G)^{H} $ is non-free, which is a contradiction. 
Thus $ \mathcal{A}(G) $ is non-free. 
\end{proof}

\begin{proposition}\label{free => III}
Let $ G $ be the signed graph in Figure \ref{Fig:obstruction}. 
Then $ \mathcal{A}(G) $ is non-free. 
\end{proposition}
\begin{proof}
The characteristic polynomial $ \mathcal{A}(G) $ is 
\begin{align*}
\chi(\mathcal{A}(G), t)
= t^{4} - 12t^{3} + 52t^{2} - 92t + 51
= (t - 1)(t - 3)(t^{2} - 8t + 17). 
\end{align*}
By Terao's factorization theorem, we have $ \mathcal{A}(G) $ is non-free. 
\end{proof}

From Proposition \ref{free => I}, \ref{free => locally free}, \ref{free => II}, and \ref{free => III}, we obtain the following lemma, that proves the implication $ (\ref{main theorem 3}) \Rightarrow (\ref{main theorem 1}) $ of Theorem \ref{main theorem}.. 
\begin{lemma}\label{3=>1}
Let $ G $ be a signed graph. 
If $ \mathcal{A}(G) $ is free, then conditions (\ref{I}), (\ref{II}), and (\ref{III}) hold. 
\end{lemma}

\section{The graph $ \tilde{G} $}\label{sec:tildeG}
In this section, we will describe how to construct a new signed graph $ \tilde{G} $ from a given signed graph $ G $. Furthermore, we will investigate the properties of $ \tilde{G} $. This new graph will play an important role in this article.
\begin{definition}
Let $ G=(G^{+},G^{-}) $ be a signed graph.
Define $ \tilde{G} \coloneqq (\tilde{G}^{+}, \tilde{G}^{-}) $,  where $ \tilde{G}^{+} \coloneqq G^{+} \cup G^{-} $ and $ \tilde{G}^{-} \coloneqq G^{+} \cap G^{-} $. 
\end{definition}
Note that $ \tilde{G}^{+} $ is the underlying simple graph of $ G $ and $ \tilde{G}^{-} $ is the simple graph consisting of the double edges of $ G $. 
Therefore $ \tilde{G} $ satisfies $ \tilde{G}^{+} \supseteq \tilde{G}^{-} $. 
Furthermore, we obtain the graph $ \tilde{G} $ by replacing the single edges of $ G $ with single positive edges. 
Hence, for any switching function $ \nu $, $ G $ and $ G^{\nu} $ yield the same graph $ \tilde{G} $. 
However, $ G $ and $ \tilde{G} $ are not switching equivalent in general (see Figure \ref{Fig:example tildeG}). 
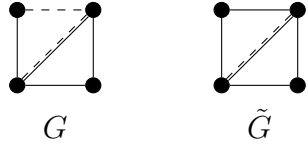
\begin{figure}[t]
\centering
\begin{tikzpicture}
\draw (0,1) node[v](1){};
\draw (0,0) node[v](2){};
\draw (1,0) node[v](3){};
\draw (1,1) node[v](4){};
\draw (1)--(2)--(3)--(4);
\draw[dashed] (1)--(4);
\draw[decoration={dashsoliddouble}, decorate] (2)--(4);
\draw (0.5,-0.5) node(5){$ G $\vphantom{$ \tilde{G} $}};
\end{tikzpicture}
\hspace{12mm}
\begin{tikzpicture}
\draw (0,1) node[v](1){};
\draw (0,0) node[v](2){};
\draw (1,0) node[v](3){};
\draw (1,1) node[v](4){};
\draw (1)--(2)--(3)--(4)--(1);
\draw[decoration={dashsoliddouble}, decorate] (2)--(4);
\draw (0.5,-0.5) node(5){$ \tilde{G} $};
\end{tikzpicture}
\caption{Example of $ G $ and $ \tilde{G} $}\label{Fig:example tildeG}
\end{figure}

\begin{proposition}\label{G^us chordal}
If a signed graph $ G $ satisfies (\ref{I}) and (\ref{II}), then $ \tilde{G}^{+} $ is chordal. 
\end{proposition}
\begin{proof}
Let $ C $ be a cycle of $ \tilde{G}^{+} $ of length at least four. 
Let $ C^{\prime} $ be a cycle of $ G $ corresponding to $ C $. 
It suffices to show that $ C^{\prime} $ has a chord since every chord of $ C^{\prime} $ leads a chord of $ C $. 
If $ C^{\prime} $ is unbalanced, then $ C^{\prime} $ has a chord or admits a double edge by (\ref{II}). 
When the latter holds, we obtain a balanced cycle by changing the sign of the double edge $ C^{\prime} $. 
Hence we may assume that $ C^{\prime} $ is balanced. 
By (\ref{I}), we have that $ C^{\prime} $ has a chord. 
\end{proof}

\begin{proposition}\label{simp tilde}
Let $ G $ be a signed graph. 
Then a link-simplicial vertex in $ G $ is link simplicial in $ \tilde{G} $. 
Moreover, when $ G $ has no induced subgraph isomorphic to an unbalanced triangle, a link-simplicial vertex in $ \tilde{G} $ is link simplicial in $ G $.  
\end{proposition}
\begin{proof}
First we assume that $ v $ is link simplicial in $ G $ and prove that $ v $ is link simplicial in $\tilde{G} $. 
Take edges $ e,e^{\prime} $ of $ \tilde{G} $ incident to $ v $ such that the other endvertices $ u,u^{\prime} $ are distinct. 

If $ e $ or $ e^{\prime} $ is double, then there exists a double edge between $ u $ and $ u^{\prime} $ since $ v $ is a link-simplicial vertex of $ G $. 
Therefore we may choose an edge forming a balanced triangle with $ e, e^{\prime} $. 

Now we assume that both $ e $ and $ e^{\prime} $ are positive single edges. 
Then we have a positive edge in $ \tilde{G} $ since $ v $ is link simplicial in $ G $ and this edge forms a balanced triangle with $ e, e^{\prime} $. 

Next suppose that $ G $ has no induced subgraph isomorphic to an unbalanced triangle and take a link-simplicial vertex $ v $ in $ \tilde{G} $. 
We prove that $ v $ is link simplicial in $ G $. 
Let $ e, e^{\prime} $ be edges of $ G $ incident to $ v $ such that the other endvertices $ u, u^{\prime} $ are different. 

If $ e $ or $ e^{\prime} $ is double, we may prove the assertion in a similar way to the above. 
Suppose that both $ e $ and $ e^{\prime} $ are single edges. 
Since $ v $ is link simplicial in $ \tilde{G} $, there exists an edge $ e^{\prime\prime} $ of $ G $ between $ u $ and $ u^{\prime} $. 
If $ e,e^{\prime},e^{\prime\prime} $ form an unbalanced triangle, then $ e^{\prime\prime} $ is double since $ G $ has no induced subgraph isomorphic to an unbalanced triangle. 
Thus we obtain an edge forming a balanced triangle with $ e $ and $ e^{\prime} $. 
Therefore $ v $ is link simplicial in $ G $. 
\end{proof}

\section{The case $ \tilde{G}^{+} $ is complete}\label{sec:tildeG is complete}
In this section, we consider the case that $ \tilde{G}^{+} $ is complete. 

\begin{proposition}\label{Gus comp I III => pm threshold}
Suppose that $ \tilde{G}^{+} $ is complete. 
If $ G $ satisfies (\ref{I}) and (\ref{III}), then $ \tilde{G}^{-} $ is threshold. 
\end{proposition}
\begin{proof}
Assume that $ \tilde{G}^{-} $ is not threshold. 
Then $ \tilde{G}^{-} $ has an induced subgraph $ H $ isomorphic to $ 2K_{2}, C_{4}, $ or $ P_{4} $ by Theorem \ref{forbidden induced subgraphs threshold}. 
Let $ V_{H} = \{a,b,c,d\} $ be the vertex set of $H$. 

First suppose that $ H = 2K_{2} $. 
We may assume that $ \{a,b\} $ and $ \{c,d\} $ are the edges of $ H $. 
In other words, $ \{a,b\} $ and $ \{c,d\} $ are the double edges of $ G[\{a,b,c,d\}] $. 
By switching, we may assume that $ \{a,d\} $ and $ \{b,c\} $ are single positive edges. 
Considering the balanced cycles $ abcda $ and $ [ab][cd]a $, by condition (\ref{I}), we may assume that $ \{a,c\} $ is a single positive edge and $ \{b,d\} $ is a single negative edge, which contradicts condition (\ref{III}). 

Next suppose that $ H = C_{4} $ and that $ \{a,b\} , \{b,c\}, \{c,d\}, \{d,a\} $ are the double edges of $ G[\{a,b,c,d\}] $. 
By switching, we may assume that the single edges $ \{a,c\} $ and $ \{b,d\} $ are positive. 
Consider the balanced cycle $ [ab][cd]a $, we have that $ \{a,c\} $ or $ \{b,d\} $ is negative, which is a contradiction. 
The case $ H = P_{4} $ is similar. 
Thus the assertion holds. 
\end{proof}

We are now ready to prove Theorem \ref{main theorem} in the case that $ \tilde{G}^{+} $ is complete.

\begin{proposition}\label{G^us complete}
Suppose that $ \tilde{G}^{+} $ is complete. 
Then the following conditions are equivalent: 
\begin{enumerate}[(1)]
\item\label{G^us complete 1} $ G $ has a link elimination ordering. 
\item\label{G^us complete 2} $ \mathcal{A}(G) $ is supersolvable. 
\item\label{G^us complete 3} $ \mathcal{A}(G) $ is free. 
\item\label{G^us complete 4} $ G $ satisfies (\ref{I}), (\ref{II}), and (\ref{III}). 
\item\label{G^us complete 5} $ G $ satisfies (\ref{II}) and $ \tilde{G}^{-} $ is threshold. 
\end{enumerate}
\end{proposition}
\begin{proof}
By Theorem \ref{zaslavsky ss}, we have $ (\ref{G^us complete 1}) \Rightarrow (\ref{G^us complete 2}) $. 
The implication $ (\ref{G^us complete 2}) \Rightarrow (\ref{G^us complete 3}) $ is well known, as mentioned before. 
From Lemma \ref{3=>1}, the implication $ (\ref{G^us complete 3}) \Rightarrow (\ref{G^us complete 4}) $ holds. 
By Proposition \ref{Gus comp I III => pm threshold}, we have $ (\ref{G^us complete 4}) \Rightarrow (\ref{G^us complete 5}) $. 

We now show the implication $ (\ref{G^us complete 5}) \Rightarrow (\ref{G^us complete 1}) $.  
Let $ (v_{1}, \dots, v_{\ell}) $ be a degree ordering of $ G $. 
Since $ \deg_{G}(v_{i}) = \deg_{\tilde{G}}(v_{i}) =  \deg_{\tilde{G}^{+}}(v_{i}) + \deg_{\tilde{G}^{-}}(v_{i}) $ for every $ i \in \{1, \dots, \ell\} $ and $ \tilde{G}^{+} $ is complete, the ordering $ (v_{1}, \dots, v_{\ell}) $ is a degree ordering of $ \tilde{G}^{-} $. 
We proceed by induction on $ \ell $. 
The case $ \ell = 1 $ is trivial. 
Assume that $ \ell \geq 2 $. 
It is sufficient to prove that $ v_{\ell} $ is a link-simplicial vertex of $ G $. 

First assume that $ v_{\ell} $ is isolated in $ \tilde{G}^{-} $. 
Then $ v_{\ell} $ is connected to the other vertices with single edges in $ G $. 
Take two distinct single edges $ \{v_{i}, v_{\ell} \} $ and $ \{v_{j},v_{\ell}\} $. 
Since $ \tilde{G}^{+} $ is complete and the induced subgraph $ G[\{v_{i},v_{j},v_{\ell}\}] $ is not isomorphic to a unbalanced cycle by condition (\ref{II}), we have that $ \{v_{i},v_{j}\} $ is a double edge or $ G[\{v_{i},v_{j},v_{\ell}\}] $ is a balanced triangle. 
In both cases, the vertex $ v_{\ell} $ is a link-simplicial vertex of $ G $. 

Next we suppose that $ v_{\ell} $ is not isolated in $ \tilde{G}^{-} $. 
Since $ \tilde{G}^{-} $ is threshold,  $ v_{1} $ is a dominating vertex of $ \tilde{G}^{-} $. 
In other words, $ v_{1} $ is connected to the other vertices by double edges in $ G $. 
By the induction hypothesis $ v_{\ell} $ is link simplicial in $ G[\{v_{2}, \dots, v_{\ell}\}] $. 
Therefore $ v_{\ell} $ is link simplicial in $ G $. 
\end{proof}


\section{Balanced chordality of $ \tilde{G} $}\label{sec:BC of tildeG}
This section is devoted to prove the following crucial lemma. 
\begin{lemma}\label{G tilde balanced chordal}
If a signed graph $ G $ satisfies conditions (\ref{I}) (\ref{II}) (\ref{III}), then $ \tilde{G} $ is balanced chordal. 
\end{lemma}

In order to prove the previous lemma, we need the following proposition. 
\begin{proposition}\label{AB double chord}
Suppose that $ G $ is a signed graph satisfying conditions (\ref{I}) (\ref{II}) (\ref{III}) and let $ C $ be a cycle of $ G $ of length four or more admitting two double edges. 
Let $ A,B $ be the vertex sets of the connected components of the graph obtained by removing these double edges from $ C $. 
If $ C $ has no chords connecting vertices within $ A $ and within $ B $, then $ C $ has a double chord between a vertex in $ A $ and a vertex in $ B $.
\end{proposition}
\begin{proof}
We will show the assertion by induction on the length of $ C $ and separate the cases. 
Without loss of generality, we may assume that all the edges of $ C $ are positive by switching. 
We label the vertices of $ C $ as shown in Figure \ref{Fig:cycle and double edges}
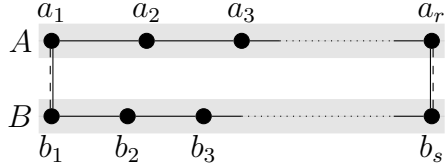
\begin{figure}[t]
\centering
\begin{tikzpicture}
\draw (0,1) node[v,label=above:{$ a_{1} $}, label=left:{$ A $}](a1){}; 
\draw (1.25,1) node[v,label=above:{$ a_{2} $}](a2){}; 
\draw (2.5,1) node[v,label=above:{$ a_{3} $}](a3){}; 
\draw (5,1) node[v,label=above:{$ a_{r} $}](ar){}; 
\draw (0,0) node[v,label=below:{$ b_{1} $}, label=left:{$ B $}](b1){};
\draw (1,0) node[v,label=below:{$ b_{2} $}](b2){};  
\draw (2,0) node[v,label=below:{$ b_{3} $}](b3){};  
\draw (5,0) node[v,label=below:{$ b_{s} $}](bs){};  
\draw[decoration={dashsoliddouble}, decorate] (b1)--(a1);
\draw[decoration={dashsoliddouble}, decorate] (ar)--(bs);
\draw (a1)--(a2)--(a3)--(3,1);
\draw[dotted] (3,1)--(4.5,1);
\draw (4.5,1)--(ar);
\draw (b1)--(b2)--(b3)--(2.5,0);
\draw[dotted] (2.5,0)--(4.5,0);
\draw (4.5,0)--(bs);
\begin{pgfonlayer}{background}
\draw[e,gray] (-.17,1)--(5.17,1);
\draw[e,gray] (-.17,0)--(5.17,0);
\end{pgfonlayer}
\end{tikzpicture}
\caption{The cycle $ C $  with the double edges}\label{Fig:cycle and double edges}
\end{figure}
\begin{figure}[t]
\centering
\begin{tikzpicture}
\draw (1.5,1) node[v,label=above:{$ a_{1} $}](a1){};
\draw (0,0) node[v,label=below:{$ b_{1} $}](b1){};
\draw (1,0) node[v,label=below:{$ b_{2} $}](b2){};  
\draw (3,0) node[v,label=below:{$ b_{s} $}](bs){};  
\draw[decoration={dashsoliddouble}, decorate] (b1)--(a1);
\draw[decoration={dashsoliddouble}, decorate] (a1)--(bs);
\draw (b1)--(b2)--(1.5,0);
\draw[dotted] (1.5,0)--(2.5,0);
\draw (2.5,0)--(bs);
\draw (1.5,-1.25) node(){(a)};
\end{tikzpicture}
\qquad
\begin{tikzpicture}
\draw (0,1) node[v,label=above:{$ a_{1} $}](a1){};
\draw (1,1) node[v,label=above:{$ a_{2} $}](a2){};
\draw (0,0) node[v,label=below:{$ b_{1} $}](b1){};
\draw (1,0) node[v,label=below:{$ b_{2} $}](b2){};
\draw[decoration={dashsoliddouble}, decorate] (b1)--(a1);
\draw[decoration={dashsoliddouble}, decorate] (a2)--(b2);
\draw (a1)--(a2);
\draw (b1)--(b2);
\draw (.5,-1.25) node(){(b)};
\end{tikzpicture}
\qquad
\begin{tikzpicture}
\draw (1,1) node[v,label=above:{$ a_{1} $}](a1){};
\draw (2,1) node[v,label=above:{$ a_{2} $}](a2){};
\draw (0,0) node[v,label=below:{$ b_{1} $}](b1){};
\draw (1,0) node[v,label=below:{$ b_{2} $}](b2){};  
\draw (3,0) node[v,label=below:{$ b_{s} $}](bs){};  
\draw[decoration={dashsoliddouble}, decorate] (b1)--(a1);
\draw[decoration={dashsoliddouble}, decorate] (a2)--(bs);
\draw[decoration={dashsoliddouble}, decorate] (a1)--(a2);
\draw (b1)--(b2)--(1.5,0);
\draw[dotted] (1.5,0)--(2.5,0);
\draw (2.5,0)--(bs);
\draw (1.5,-1.25) node(){(c)};
\end{tikzpicture}
\caption{Cases 1, 2, and 3}\label{Fig:case123}
\end{figure}
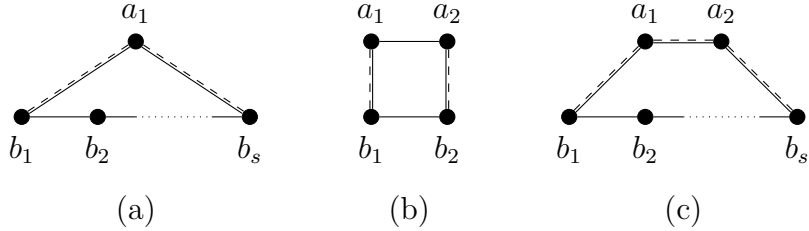
\begin{figure}[t]
\centering
\begin{tikzpicture}
\draw (0,1) node[v,label=above:{$ a_{1} $}](a1){}; 
\draw (1.25,1) node[v,label=above:{$ a_{n} $}](an){}; 
\draw (2.5,1) node[v,label=above:{$ a_{m} $}](am){}; 
\draw (3.75,1) node[v,label=above:{$ a_{r} $}](ar){}; 
\draw (0,0) node[v,label=below:{$ b_{1} $}](b1){};
\draw (1,0) node[v,label=below:{$ b_{2} $}](b2){};  
\draw (3.75,0) node[v,label=below:{$ b_{s} $}](bs){};  
\draw[decoration={dashsoliddouble}, decorate] (b1)--(a1);
\draw[decoration={dashsoliddouble}, decorate] (ar)--(bs);
\draw (a1)--(0.375,1);
\draw[dotted] (0.375,1)--(.825,1);
\draw (.825,1)--(an)--(1.625,1);
\draw[dotted] (1.625,1)--(2.125,1);
\draw (2.125,1)--(am)--(2.825,1);
\draw[dotted] (2.825,1)--(3.375,1);
\draw (3.375,1)--(ar);
\draw (b1)--(b2)--(1.325,0);
\draw[dotted] (1.325,0)--(3.375,0);
\draw (3.375,0)--(bs);
\draw (an)--(b1);
\draw (am)--(bs);
\draw (1.875,-1.25) node(){(a)};
\end{tikzpicture}
\quad
\begin{tikzpicture}
\draw (0,1) node[v,label=above:{$ a_{1} $}](a1){}; 
\draw (1.25,1) node[v,label=above:{$ a $}](a){}; 
\draw (2.5,1) node[v,label=above:{$ a^{\prime} $}](ap){}; 
\draw (3.75,1) node[v,label=above:{$ a_{r} $}](ar){}; 
\draw (0,0) node[v,label=below:{$ b_{1} $}](b1){};
\draw (1,0) node[v,label=below:{$ b $}](b){};  
\draw (2.75,0) node[v,label=below:{$ b^{\prime} $}](bp){};  
\draw (3.75,0) node[v,label=below:{$ b_{s} $}](bs){};  
\draw[decoration={dashsoliddouble}, decorate] (b1)--(a1);
\draw[decoration={dashsoliddouble}, decorate] (ar)--(bs);
\draw (a1)--(0.375,1);
\draw[dotted] (0.375,1)--(.825,1);
\draw (.825,1)--(a)--(1.625,1);
\draw[dotted] (1.625,1)--(2.125,1);
\draw (2.125,1)--(ap)--(2.825,1);
\draw[dotted] (2.825,1)--(3.375,1);
\draw (3.375,1)--(ar);
\draw (b1)--(0.375,0);
\draw[dotted] (0.375,0)--(0.675,0);
\draw (0.675,0)--(b)--(1.325,0);
\draw[dotted] (1.325,0)--(2.325,0);
\draw (2.325,0)--(bp)--(3.125,0);
\draw[dotted] (3.125,0)--(3.375,0);
\draw (3.375,0)--(bs);
\draw (a)--(b);
\draw[dashed] (ap)--(bp);
\draw (1.875,-1.25) node(){(b)};
\end{tikzpicture}
\quad
\begin{tikzpicture}
\draw (0,1) node[v,label=above:{$ a_{1} $}](a1){}; 
\draw (1,1) node[v,label=above:{$ a_{m} $}](am){}; 
\draw (2.5,1) node[v,label=above:{$ a_{i} $}](ai){}; 
\draw (3.75,1) node[v,label=above:{$ a_{r} $}](ar){}; 
\draw (0,0) node[v,label=below:{$ b_{1} $}](b1){};
\draw (1.75,0) node[v,label=below:{$ b_{j-1} $}](bj-1){};  
\draw (2.75,0) node[v,label=below:{$ b_{j} $}](bj){};  
\draw (3.75,0) node[v,label=below:{$ b_{s} $}](bs){};  
\draw[decoration={dashsoliddouble}, decorate] (b1)--(a1);
\draw[decoration={dashsoliddouble}, decorate] (ar)--(bs);
\draw[decoration={dashsoliddouble}, decorate] (bj)--(bj-1);
\draw (a1)--(0.375,1);
\draw[dotted] (0.375,1)--(.675,1);
\draw (.675,1)--(am)--(1.325,1);
\draw[dotted] (1.325,1)--(2.125,1);
\draw (2.125,1)--(ai)--(2.825,1);
\draw[dotted] (2.825,1)--(3.375,1);
\draw (3.375,1)--(ar);
\draw (b1)--(0.375,0);
\draw[dotted] (0.375,0)--(1.375,0);
\draw (1.375,0)--(bj-1);
\draw (bj)--(3.125,0);
\draw[dotted] (3.125,0)--(3.375,0);
\draw (3.375,0)--(bs);
\draw[dashed] (ai)--(bj-1);
\draw (ai)--(bj)--(am);
\draw (2.4,.48) node(){\scalebox{.7}{$ T $}};
\draw (1.875,-1.25) node(){(c)};
\end{tikzpicture}
\caption{Case 4}\label{Fig:case4}
\end{figure}
\begin{case}\label{case1}
Assume that $ |A| = 1 $ (Figure \ref{Fig:case123}a). 
By induction, we will show that every chord $ \{a_{1},b_{i}\} \quad (2 \leq i \leq s-1) $ is double. 

First assume that the length of $ C $ equals four, that is, $ s=3 $. 
By condition (\ref{I}), the balanced cycle $ a_{1}b_{1}b_{2}b_{3}a_{1} $ yields a positive chord $ \{a_{1},b_{2}\} $. 
Similarly, considering the balanced cycle $ [a_{1}b_{1}]b_{2}[b_{3}a_{1}] $, we obtain a negative chord $ \{a_{1}, b_{2}\} $. 
Thus $ \{a_{1},b_{2}\} $ is a double chord. 

Next, we assume that $ s \geq 4 $. 
By condition (\ref{I}), we obtain a positive chord $ \{a_{1},b_{i}\} $ for some $ i \in \{2, \dots, s-1\} $, which separates our cycle $ C $ into two smaller cycles. 
By the induction hypothesis, we may conclude that the assertion is true. 
\end{case}
\begin{case}
Assume that the length of $ C $ equals four. 
If $ |A|=1 $ or $ |B|=1 $, then the assertion holds by Case \ref{case1}. 
Hence we may assume that $ |A|=|B|=2 $ (Figure \ref{Fig:case123}b). 
The balanced cycles $ a_{1}b_{1}b_{2}a_{2}a_{1} $ and $ [a_{1}b_{1}][b_{2}a_{2}]a_{1} $ yield a positive and a negative chord, respectively. 
If these two chords coincide, then we obtain a desired double chord. 
If not, the simple graph $ \widetilde{G[A \cup B]}^{+} $ is complete. 
Therefore, by Proposition \ref{G^us complete}, we have that the simple graph $ \widetilde{G[A \cup B]}^{-} $ is threshold. 
By Theorem \ref{forbidden induced subgraphs threshold}, one of $ \{a_{1},b_{2}\} $ and $ \{a_{2},b_{1}\} $ is double. 
\end{case}
\begin{case}\label{case3}
Assume that $ |A|=2 $ and $ \{a_{1},a_{2}\} $ is double (Figure \ref{Fig:case123}c). 
Then the balanced cycle $ a_{1}b_{1}\cdots b_{s}a_{2}a_{1} $ yields a positive chord between a vertex in $ A $ and a vertex in $ B $. 
By symmetry, we may assume that one of the endvertices is $ a_{1} $. 
Take a positive chord $ \{a_{1},b_{m}\} $ with $ m $ maximal. 
If $ m < s $, then the cycle $ a_{1}b_{m}b_{m+1}\cdots b_{s}a_{2}a_{1} $ is of length four or more admitting double edges $ \{a_{1},a_{2}\} $ and $ \{a_{2},b_{2}\} $. 
Applying the result of Case \ref{case1}, every $ \{a_{2},b_{i}\} \quad (m \leq i \leq s) $ is double. 
Therefore we may assume that $ m=s $, that is, there exists the positive chord $ \{a_{1},b_{s}\} $. 
Since the cycle $ a_{1}b_{1}\cdots b_{s}a_{1} $ is balanced, every $ \{a_{1},b_{i}\} \quad (2 \leq i \leq s) $ is a positive chord. 
Considering the balanced cycle $ [a_{1}b_{1}]b_{2}\cdots b_{s}[a_{2}a_{1}] $, we obtain a balanced chord. 
If one of the endvertices of this chord is $ a_{1} $, then this chord is negative and hence we have a double chord. 
Therefore we may assume that there is a positive chord between $ a_{2} $ and a vertex in $ B $. 
Let $ \{a_{2},b_{m}\} $ be a positive chord with $ m $ minimal. 
By the similar discussion above, we may assume that $ m=1 $ and we obtain positive chords $ \{a_{2}, b_{i}\} \quad (1 \leq i \leq s) $. 
The balanced cycle $ [a_{1}b_{1}]b_{2}\cdots b_{s-1} [b_{s}a_{2}]a_{1} $ yields a negative chord, which must be double. 
\end{case}
\begin{case}
Assume that the length of $ C $ is at least five. 
By Case \ref{case1}, we may suppose that $ |A| \geq 2 $ and $ |B| \geq 3 $. 
First, we show that there exists a positive chord $ \{a,b\} $ such that $ a \in A $ and $ b \in B\setminus\{b_{1},b_{s}\} $. 
By condition (\ref{I}), there exists a positive chord between a vertex in $ A $ and a vertex in $ B $. 
We may assume that one of the endvertices of this chord is $ b_{s} $. 
Let $ \{a_{m},b_{s}\} $ be the positive chord with $ m $ minimal. 
Since $ s = |B| \geq 3 $, the balanced cycle $ a_{1}b_{1}\cdots b_{s}a_{m}\cdots a_{1} $ is of length four or more and hence has a positive chord. 
We may assume that one of the endvertices of this chord is $ b_{1} $ and take the positive chord $ \{a_{n},b_{1}\} $ such that $ n \leq m $ and $ n $ is maximal (Figure \ref{Fig:case4}a). 
The balanced cycle $ a_{n}b_{1}\cdots b_{s}a_{m} \cdots a_{n} $ is of length four or more and hence we obtain a desired positive chord $ \{a,b\} $. 
Using the balanced cycle $ [a_{1}b_{1}]b_{2}\cdots b_{s-1}[b_{s}a_{r}]\cdots a_{1} $, we have a negative chord $ \{a^{\prime},b^{\prime}\} $ such that $ a^{\prime} \in A $ and $ b^{\prime} \in B \setminus \{b_{1},b_{s}\} $ in a similar way. 

If the positive chord $ \{a,b\} $ and the negative chord $ \{a^{\prime},b^{\prime}\} $ coincide, then this is a desired double chord. 
Otherwise, the positive chord $ \{a,b\} $, the negative chord $ \{a^{\prime},b^{\prime}\} $, the path between $ a $ and $ a^{\prime} $, and the path between $ b $ and $ b^{\prime} $ form an unbalanced cycle (Figure \ref{Fig:case4}b). 
If the length of this unbalanced cycle is four or more, then it has a chord by Proposition \ref{G^us chordal}. 
This chord separates the unbalanced cycle into smaller balanced and unbalanced cycles. 
Applying the same argument to the smaller unbalanced cycle one after another, we obtain a unbalanced triangle $ T $ consisting of a positive and a negative chords of $ C $ and an edge $ e $ lying in $ A $ or $ B $. 
If one of these two chords is double, then this is a desired double chord. 
Hence we may assume that $ e $ is double by condition (\ref{II}). 

If $ |A|=2 $ and $ e = \{a_{1},a_{2}\} $, then we have a double chord by Case \ref{case3}. 
Hence we may assume that $ |A| \geq 3 $ or the edge $ e $ lies in $ B $. 
In both cases, without loss of generality, we may assume that the vertex set of $ T $ is $ \{a_{i},b_{j-1},b_{j}\} $ for some $ i \in \{1,\dots,r\} $ and $ j \in \{3,\dots,s\} $ by symmetry (Figure \ref{Fig:case4}c). 
Let $ m $ be the minimal number such that the chord $ \{a_{m},b_{j}\} $ exists. 
By switching, we may assume that this chord is positive. 
The balanced cycle $ a_{1}b_{1} \cdots b_{j}a_{m} \cdots a_{1} $ is of length four or more and satisfies the hypothesis of this proposition. 
By the induction hypothesis, this cycle has a double chord, which is a desired double chord of $ C $. 
\end{case}
\end{proof}
\begin{proof}[Proof of Lemma \ref{G tilde balanced chordal}]
Let $ C $ be a balanced cycle of $ \tilde{G} $ of length at least four. 
If every edge of $ C $ is positive, then $ C $ has a positive chord by Proposition \ref{G^us chordal}. 
Hence we may assume that $ C $ has at least two negative edges. 
Let $ C^{\prime} $ be a cycle of $ G $ corresponding to $ C $. 
Since $ C $ has at least two negative edges, the cycle $ C^{\prime} $ admits at least two double edges. 
Hence we may assume that $ C^{\prime} $ is a cycle as shown in Figure \ref{Fig:cycle and double edges}. 
Without loss of generality, we may assume that $ |A| \leq |B| $. 

Let $ C^{\prime\prime} $ be a cycle of $ G $ on the vertices $ a_{1},b_{1},a_{r},b_{s}, $ and some vertices in $ A $ and $ B $ such that there is no chord lying in both $ A $ and $ B $. 
Suppose that the length of $ C^{\prime\prime} $ is three. 
This happens only when $ |A|=1 $ and $ C^{\prime\prime} $ is a cycle on the vertices $ \{a_{1},b_{1},b_{s}\} $. 
The edge $ \{b_{1},b_{s}\} $ leads to a positive chord of $ C $. 
Therefore we may assume that the length of $ C^{\prime\prime} $ is at least four. 
By Proposition \ref{AB double chord}, the cycle $ C^{\prime\prime} $ has a double chord between a vertex in $ A $ and a vertex in $ B $, which is also a double chord of $ C^{\prime} $. 
Choosing a sign of the corresponding double chord of $ C $, we may say that $ C $ has a balanced chord. 
Thus $ \tilde{G} $ is balanced chordal. 
\end{proof}

\section{Proof of Theorem \ref{main theorem}}\label{sec:PROOF}
The following result played a key role in \cite{suyama2019signed-dm}. 
\begin{proposition}[{\cite[Proposition 6.5 and Lemma 6.9]{suyama2019signed-dm}}]\label{STT linksimp-gluing}
Let $ G $ be a signed graph with $ G^{+} \supseteq G^{-} $. 
If $ G $ is balanced chordal, then one of the following holds: 
\begin{enumerate}[(i)]
\item\label{i} $ G $ has a link-simplicial vertex. 
\item\label{ii} There exist induced subgraphs $ G_{1} $ and $ G_{2} $ such that $ G_{1} \cup G_{2} = G $ and $ G_{1} \cap G_{2} = K^{\pm}_{n} $, where $ K^{\pm}_{n} \coloneqq (K_{n}, K_{n}) $ denotes a signed graph on $ n $ vertices with all possible edges, which is called the \textbf{complete signed graph}. 
\end{enumerate}
\end{proposition}

If $ G $ is a signed graph with property (\ref{ii}) in Proposition \ref{STT linksimp-gluing}, then the characteristic polynomial of $ \mathcal{A}(G) $ is decomposed as 
\begin{align*}
\chi(\mathcal{A}(G), t) = \dfrac{\chi(\mathcal{A}(G_{1}),t) \, \chi(\mathcal{A}(G_{2}),t)}{\chi(\mathcal{A}(K^{\pm}_{n}),t)}
\end{align*}
as shown in \cite[Theorem 4.10 and Lemma 4.13]{suyama2019signed-dm}. 
The decomposition is a special case of a property of generalized parallel connections of simple matroids studied by Brylawski \cite{brylawski1975modular-totams}. 
In \cite{tsujie2020modular-a}, the second author investigated some arrangements obtained by modular joins (a special kind of generalized parallel connections). 
For signed graphs, we have the following proposition. 

\begin{proposition}[{\cite[Theorem 4.8]{tsujie2020modular-a}}]\label{modular join}
Let $ \mathfrak{M} $ be the minimal class of signed graphs satisfying the following conditions. 
\begin{enumerate}[(i)]
\item The null graph is a member of $ \mathfrak{M} $. 
\item If $ G $ has a link-simplicial vertex $ v $ and $ G\setminus v \in \mathfrak{M} $, then $ G \in \mathfrak{M} $. 
\item If there exist induced subgraphs $ G_{1} $ and $ G_{2} $ of a signed graph of $ G $ such that $ G_{1} \cup G_{2} = G $ and $ G_{1} \cap G_{2} = K^{\pm}_{n} $ for some $ n $ and $ G_{1}, G_{2} \in \mathfrak{M} $, then $ G \in \mathfrak{M} $. 
\end{enumerate}
Then for every $ G \in \mathfrak{M} $, the corresponding arrangement $ \mathcal{A}(G) $ is divisionally free. 
\end{proposition}

\begin{remark}
In this article, we always suppose that $ \mathcal{A}(G) $ contains the Boolean arrangements. 
Therefore the graph is considered to equip all loops in terminology of \cite{tsujie2020modular-a}. 
\end{remark}

Now we are ready to prove the main theorem. 
\begin{proof}[Proof of Theorem \ref{main theorem}]
The implication $ (\ref{main theorem 3}) \Rightarrow (\ref{main theorem 1}) $ is due to Lemma \ref{3=>1}. 
The implication $ (\ref{main theorem 2}) \Rightarrow (\ref{main theorem 3}) $ is trivial by the definition of divisional freeness. 

We just need to prove the implication $ (\ref{main theorem 1}) \Rightarrow (\ref{main theorem 2})$. 
This will be achieved by induction on $ \ell $, the number of vertices of $ G $. 
We may assume that $ \ell \geq 2 $ and show that $ G \in \mathfrak{M} $, where $ \mathfrak{M} $ is defined in Proposition \ref{modular join}. 

By Lemma \ref{G tilde balanced chordal}, $ \tilde{G} $ is balanced chordal. 
Therefore $ \tilde{G} $ has a link-simplicial vertex $ v $ or $ \tilde{G} $ is the union of two signed graphs whose intersection is a complete signed graph. 

Suppose the former case holds. 
By Proposition \ref{simp tilde}, the vertex $ v $ is link simplicial in $ G $. 
Since the induced subgraph $ G\setminus v $ also satisfies conditions (\ref{I})(\ref{II})(\ref{III}), the graph $ G\setminus v $ belongs to $ \mathfrak{M} $ bu the induction hypothesis. 
Hence $ G $ also belongs to $ \mathfrak{M} $. 

Next we consider the latter case. 
Then there exist induced subgraphs $ G_{1} $ and $ G_{2} $ such that $ G_{1}\cup G_{2}=G $ and $ G_{1} \cap G_{2} = K^{\pm}_{n} $ for some $ n $. 
By the induction hypothesis $ G_{1}, G_{2} \in \mathfrak{M} $ and hence $ G \in \mathfrak{M} $. 

Thus we have proven that $ G \in \mathfrak{M} $ by induction. 
Applying Proposition \ref{modular join}, we can conclude that $ \mathcal{A}(G) $ is divisionally free. 
\end{proof}

\bibliographystyle{amsplain}
\bibliography{bibfile}

\end{document}